\documentclass[a4paper, reqno, 11pt]{amsart}
\makeatletter

\@addtoreset{equation}{section}
\makeatother
\usepackage{setspace}
\usepackage{xcolor}
\usepackage{braket}
\usepackage{amsmath}
\usepackage{amssymb}
\usepackage{latexsym}
\usepackage{amsthm}
\usepackage{textcomp}
\usepackage{hyperref}
\usepackage{here}
\newtheorem{Thm}{Theorem}[section]

\newtheorem{Lem}[Thm]{Lemma}
\newtheorem{Prop}[Thm]{Proposition}

\newtheorem{Cor}[Thm]{Corollary}

\theoremstyle{definition}

\newtheorem{Def}[Thm]{Definition}
\newtheorem{Rem}[Thm]{Remark}
\newtheorem{Exa}[Thm]{Example}

\begin{document}

\title[PM of r.v. and characterization via frac. calc.]{Power means of random variables and characterizations of distributions via fractional calculus}

\author[K. Okamura]{Kazuki Okamura}
\address{Department of Mathematics, Faculty of Science, Shizuoka University}
\email{okamura.kazuki@shizuoka.ac.jp}

\author[Y. Otobe]{Yoshiki Otobe}
\address{Department of Mathematics, Faculty of Science, Shinshu University}
\email{otobe@math.shinshu-u.ac.jp}

\date{\today}
\maketitle

\begin{abstract}
We investigate fractional moments and expectations of power means of complex-valued random variables by using fractional calculus.
We deal with both negative and positive orders of the fractional derivatives.
The one-dimensional distributions are characterized in terms of the fractional moments without any moment assumptions.
We explicitly compute the expectations of the power means for both the univariate Cauchy distribution and the Poincar\'e distribution on the upper-half plane.
We show that for these distributions the expectations are invariant with respect to the sample size and the value of the power. 
\end{abstract}

\section{Introduction}

The strong law of large numbers is fundamental to probability.
This law states that the arithmetic mean of independent and identically distributed (i.i.d.) random variables converges almost surely to a constant.
This constant is equal to the expectation of the respective random variable.
For stochastic processes,
once we establish the law of large numbers, then, we move on to more sophisticated limit theorems such as the central limit theorem.
For non-integrable i.i.d. random variables such as Cauchy distributions,
the law of large numbers fails for the arithmetic mean.
In order to establish universality for non-integrable i.i.d. random variables,
we need to consider an alternative statistic other than the arithmetic mean. 

The framework of {\it quasi-arithmetic means} provides an alternative approach to universality.
This notion was introduced independently by Kolmogorov \cite{Kolmogorov1930}, Nagumo \cite{Nagumo1930}, and de Finetti \cite{Finetti1931}.
They proposed axioms of means
and showed that if the axioms of means hold for an $n$-ary operation on an interval of real numbers,
then that operation must take the form of a quasi-arithmetic mean.
It is defined by
$M^f_n = M^f_n (x_1, \dots, x_n) := f^{-1}\left( \frac{1}{n} \sum_{j=1}^{n} f(x_j) \right)$
for a generator $f$.
$M^f_n$ is the arithmetic mean if $f(x) = x$.
This framework includes geometric, harmonic, and power means.
The random variable $M^f_n (X_1, \dots, X_n)$, where $X_1, \dots, X_n$ are i.i.d. random variables, can be integrable even if $X_1$ is non-integrable, and it plays a similar role to the arithmetic mean of i.i.d. random variables. 
Limit theorems such as the strong law of large numbers and the central limit theorem hold.
The strong law of large numbers states that $\lim_{n \to \infty} M^f_n = f^{-1}(E[f(X_1)])$ almost surely and it follows directly from the usual strong law of large numbers if $f(X_1) \in L^1$.
The central limit theorem is more complicated, but it follows from the delta method (Carvalho \cite{Carvalho2016}).
Recently, limit theorems of more general means of i.i.d. random variables have been considered by Barczy and Burai \cite{Barczy2022} and Barczy and P\'ales \cite{Barczy2023}.
By considering {\it complex-valued} quasi-arithmetic means of i.i.d. random variables,
we can deal with heavy-tailed random variables {\it supported on $\mathbb{R}$}.
Akaoka and the authors \cite{Akaoka2022} considered integrability and asymptotic variances of quasi-arithmetic means and gave applications to quasi-arithmetic means of Cauchy distributions.

The expectations of the arithmetic mean of i.i.d. integrable random variables are equal to the expectation of the respective random variable.
However, the expectation of $M^f_n$, $E\left[M^f_n\right]$, is difficult to compute explicitly,
because $M^f_n$ is defined by using $f^{-1}$, which can be non-linear.
In this paper, we consider $E\left[M^f_n \right]$
for the case of complex-valued {\it power means}, more specifically, the generator is given by $f(z) = z^p$, $p \in [-1,1], p \ne 0$.
Each power mean is a homogeneous quasi-arithmetic mean.
The class of power means for $p \in (-1,0)$ interpolates between the harmonic mean and the geometric mean, and, the class for $p \in (0,1)$ interpolates between the geometric mean and the arithmetic mean.

We deal with the fractional moment $E\left[Z^{\lambda}\right]$ of a complex-valued random variable $Z$ with a complex power $\lambda$,
and furthermore the expectation of the power mean $E\left[\left( \frac{1}{n} \sum_{j=1}^{n} Z_j^p \right)^{1/p}\right]$ of i.i.d. complex-valued random variables  $(Z_j)_j$ for $p \in [-1,1]$.
We explain the reason why we use {\it fractional calculus} to compute them. 
Let $\varphi_Z (t) := E[\exp(it Z)]$.
If $E\left[\left|Z\right|^n\right] < +\infty$ for some $n \in \mathbb{N}$, then,
\[ \left.\frac{\partial^n}{\partial t^n} \varphi_Z (t)\right|_{t=0} = i^n E\left[Z^n\right]. \]
We replace the natural number $n$, the order of the derivative, with a {\it fractional} number $\lambda$ and then we formally obtain that 
\[ \left.\frac{\partial^{\lambda}}{\partial t^{\lambda}} \varphi_Z (t)\right|_{t=0} = i^{\lambda} E\left[Z^{\lambda}\right]. \]

Part of our paper is devoted to justifying this heuristic argument.
There are several different frameworks for fractional derivatives (Oldham and Spanier \cite{Oldham1974}).
Here we adopt the Riemann-Liouville integral for $\textup{Re}(\lambda) < 0$ and Marchaud's fractional derivative (\cite{Marchaud1927}) for $\textup{Re}(\lambda) > 0$.
Then we proceed to consider the case that $Z = \frac{1}{n} \sum_{j=1}^{n} Z_j^p$ and $\lambda = p \in [-1,1]$.
The relationship between fractional calculus and the fractional {\it absolute} moment $E\left[|X|^p\right]$ of a real-valued random variable $X$ and $p \in \mathbb{R}$ has been considered in many papers.
For more details, see Matsui and Pawlas \cite{Matsui2016} and references therein.
Here we consider connections between fractional calculus and fractional (non-absolute) moments of complex-valued random variables.

By using our results for the fractional moment,
we show that certain subfamilies of the expectations $\left\{E\left[(X+\alpha)^{\lambda}\right] : \alpha \in \overline{\mathbb{H}}, \lambda \in \mathbb{C} \right\}$ characterize the distribution of a real-valued random variable $X$, where  $\mathbb{H}$ is the upper-half plane and $\overline{\mathbb{H}}$ is its closure. 
Our results are partly similar to those of Lin \cite{Lin1992}, but with notable differences.
We consider $X + \alpha$ instead of $X$, which makes the considerations for $\textup{Re}(\lambda) < 0$ much clearer, since with this modification we do not need to impose any integrability conditions on $X$.
Furthermore, the notion of determining sets of holomorphic functions is involved.

Our framework is applicable to {\it Poincar\'e distributions}, which are a parametric family of distributions supported on $\mathbb{H}$ recently introduced by Tojo and Yoshino \cite{tojo-2020, tojo2021harmonic}. 
If $Z$ follows a Poincar\'e distribution, then, we can explicitly compute $\varphi_Z$. 
By using our results for the fractional moment,
we can compute $E\left[Z_1^{\lambda}\right]$ and $E\left[\left( \frac{1}{n} \sum_{j=1}^{n} Z_j^p \right)^{1/p}\right]$ if $Z_1$ follows a Poincar\'e distribution.
Similarly, we can deal with the case that $Z_j = X_j + \alpha$ where $\alpha \in \mathbb{H}$ and $(X_j)_j$ are i.i.d. real-valued random variables following a Cauchy distribution or a $t$-distribution with $3$ degrees of freedom.
For the Poincar\'e and Cauchy distributions,
$E\left[\left( \frac{1}{n} \sum_{j=1}^{n} Z_j^p \right)^{1/p}\right]$ does not depend on the sample size $n$ or the parameter $p$.
However, this fails for the $t$-distribution with $3$ degrees of freedom.

The rest of the paper is organized as follows.
In Sections \ref{sec:neg} and \ref{sec:pos},
we give integral expressions for the fractional moment $E\left[Z^{\lambda}\right]$ and the expectation of the  power mean $E\left[\left( \frac{1}{n} \sum_{j=1}^{n} Z_j^p \right)^{1/p}\right]$ by using $\varphi_Z$.
The cases that $\textup{Re}(\lambda) < 0$ and $\textup{Re}(\lambda) > 0$ are considered in Sections \ref{sec:neg} and \ref{sec:pos} respectively.
In Section \ref{sec:char}, by using the results in Sections \ref{sec:neg} and \ref{sec:pos},
we show that certain subfamilies of $\{E[(X+\alpha)^{\lambda}]\}_{\alpha, \lambda}$ characterize the distribution of a real-valued random variable $X$.
In Section \ref{sec:appl}, we compute the expectations of the complex-valued power means of Cauchy distributions, $t$-distributions with 3 degrees of freedom, and Poincar\'e distributions.
In the Appendix, we compare the fractional absolute moment with the absolute value of the fractional moment.

\subsection{Notation}

We denote the imaginary unit by $i$.
For $z \in \mathbb{C}, z \ne 0$, we let
$\log z := \log |z| + i\theta$, where $z = r\exp(i\theta)$, $-\pi < \theta \le \pi$ and $r > 0$.
For $\lambda \in \mathbb{C}$,
let
$$z^{\lambda} := \begin{cases}
	\exp(\lambda \log z),& z \ne 0 \\
	0,& z = 0 \end{cases}.$$
We denote the real and imaginary parts of a complex number $z$ by $\textup{Re}(z)$ and $\textup{Im}(z)$ respectively.

Let $\overline{A}$ be the closure of a subset $A$ of $\mathbb{C}$.
Let
$\mathbb H := \{x+yi : x \in \mathbb{R}, y > 0\}$, and
$-\mathbb{H} := \{x+yi : x \in \mathbb{R}, y < 0\}$.
Let
$U := i \mathbb{H} = \left\{x+yi \mid x < 0, y \in \mathbb{R} \right\}$ and $V := -U = -i \mathbb{H} = \left\{  x+yi \mid x > 0, y \in \mathbb{R} \right\}$. 

Let the Gamma function be
$$\Gamma(\lambda) := \int_0^{\infty} t^{\lambda-1} \exp(-t) dt, \ \ \lambda \in V.$$  
For a complex-valued random variable $Z$, we denote the distribution of $Z$ by $P^Z$.
For $r > 0$, we say that $Z \in L^r$ if $|Z| \in L^r$.
For every $\alpha \in \mathbb{C}$,
$Z \in L^r$ if and only if $Z + \alpha \in L^r$.

\section{Fractional derivative of negative order}\label{sec:neg}

We use the Riemann-Liouville integral as in Wolfe  \cite{Wolfe1975} and Cressie and Borkent  \cite[Definition 1]{Cressie1986}.

\begin{Def}
Let $\lambda \in U$.
For a Borel measurable function $f : (-\infty, 0] \to \mathbb{C}$,
\[ \left.\frac{\partial^{\lambda}}{\partial t^{\lambda}} f(t)\right|_{t=0} := \frac{1}{\Gamma(-\lambda)} \int_0^{\infty} u^{-\lambda - 1} f(-u) du \]
if the integral on the right-hand side in the above display exists.
\end{Def}

It suffices to define the left derivative at $t=0$ only. 
As indicated in \cite{Wolfe1975},
it is natural to consider the fractional derivative of complex order.
Here the negative order means that the real part of the complex order is negative.

\begin{Lem}\label{lem:negative-basic-1}
Let $\lambda \in U$.
Then we have the following assertions:\\
(i) Let $Z$ be an $\mathbb{H}$-valued random variable.
Assume that $E[\textup{Im}(Z)^{\textup{Re}(\lambda)}] < +\infty$.
Then, $Z^{\lambda} \in L^1$ and
\begin{equation}\label{eq:negative-upperhalf-1}
\left.\frac{\partial^{\lambda}}{\partial t^{\lambda}} E[\exp(-itZ)]\right|_{t=0} = i^{-\lambda} E\left[Z^{\lambda}\right].
\end{equation} \\
(ii) Let $Z$ be a $(-\mathbb{H})$-valued random variable.
Assume that $E[(-\textup{Im}(Z))^{\textup{Re}(\lambda)}] < +\infty$.
Then, $Z^{\lambda} \in L^1$ and
\[ \left.\frac{\partial^{\lambda}}{\partial t^{\lambda}} E[\exp(itZ)]\right|_{t=0} = (-i)^{-\lambda} E\left[Z^{\lambda}\right]. \]
\end{Lem}

We remark that $\textup{Re}(\lambda) < 0$.

\begin{proof}
We show (i).
The proof of (ii) is similar to it.
Let $z = re^{i\theta} \in \mathbb{H}$ such that $0 < \theta < \pi$.
Then,
\[ \int_0^{\infty} |u^{-\lambda - 1} \exp(iuz)| du = \int_0^{\infty} u^{-\textup{Re}(\lambda) - 1} \exp(-u\textup{Im}(z)) du = \textup{Im}(z)^{\textup{Re}(\lambda)} \Gamma\left(-\textup{Re}(\lambda)\right). \]
By this and the assumption,
\[ \int_{\mathbb{H}} \int_0^{\infty} | u^{-\lambda - 1} \exp(iuz)| du P^Z (dz) = E\left[\textup{Im}(Z)^{\textup{Re}(\lambda)}\right] \Gamma(-\textup{Re}(\lambda)) < +\infty. \]
Therefore we can use Fubini's theorem and we obtain that
\[ \left. \frac{\partial^{\lambda}}{\partial t^{\lambda}} E[\exp(-itZ)]\right|_{t=0} =  \frac{1}{\Gamma(-\lambda)} \int_0^{\infty} u^{-\lambda - 1} E[\exp(iuZ)] du  \]
\[=  \frac{1}{\Gamma(-\lambda)} E\left[ \int_0^{\infty} u^{-\lambda - 1} \exp(iuZ) du\right].\]
By the Cauchy integral theorem, 
\begin{align}\label{eq:basic-neg}
\int_0^{\infty} u^{-\lambda - 1} \exp(iuz) du
&= z^{\lambda} \int_{C_z} \zeta^{-\lambda - 1} \exp(i\zeta) d\zeta
= z^{\lambda} \int_{C_{i}} \zeta^{-\lambda - 1} \exp(i\zeta) d\zeta \notag\\
&= z^{\lambda}  i^{-\lambda} \int_0^{\infty} u^{-\lambda-1} \exp(-u) du
= z^{\lambda} i^{-\lambda} \Gamma(-\lambda),
\end{align}
where we let $C_w := \{tw | t \ge 0\}$ for $w \in \mathbb{H}$.
Hence we have that
\begin{equation}\label{eq:complex-power-upper-abs}
|z^{\lambda}| \le \frac{1}{|i^{-\lambda} \Gamma(-\lambda)|} \int_0^{\infty} |u^{-\lambda - 1} \exp(iuz)| du =  \frac{\Gamma\left(-\textup{Re}(\lambda)\right)}{|i^{-\lambda} \Gamma(-\lambda)|} \textup{Im}(z)^{\textup{Re}(\lambda)}.
\end{equation}
By this and the assumption, $Z^{\lambda} \in L^1$, and we see \eqref{eq:negative-upperhalf-1}. 
\end{proof}

We can generalize \eqref{eq:complex-power-upper-abs} to the case that $z \in -\mathbb{H}$.
For $\lambda \in U$, there exists a constant $C(\lambda)$ such that for $z \in \mathbb{C} \setminus \mathbb{R}$,
\begin{equation}\label{eq:complex-power-abs}
|z^{\lambda}| \le C(\lambda) |\textup{Im}(z)|^{\textup{Re}(\lambda)}.
\end{equation}
E.g., we can put $$C(\lambda) := \frac{\Gamma\left(-\textup{Re}(\lambda)\right) \exp(\pi |\textup{Im}(\lambda)|/2)}{|\Gamma(-\lambda)|}.$$

Now we state an application of Lemma \ref{lem:negative-basic-1}.

\begin{Prop}\label{prop:integral-neg-fractional}
Let $X$ be a real-valued random variable and $\varphi_X$ be the characteristic function of $X$.
Let $\lambda \in U$.
Let $\alpha \in \mathbb{H}$.
Then,
$(X+\alpha)^{\lambda} \in L^{\infty}$ and
\[ E\left[ (X+\alpha)^{\lambda} \right]
= \frac{i^{\lambda}}{\Gamma(-\lambda)} \int_0^{\infty}  t^{-\lambda-1} \varphi_X (t) \exp(i\alpha t) dt.  \]
\end{Prop}

\begin{proof}
We see that by \eqref{eq:complex-power-abs},
$$\sup_{x \in \mathbb{R}} \left|(x + \alpha)^{\lambda}\right| \le C(\lambda) |\textup{Im}(\alpha)|^{\textup{Re}(\lambda)}.$$ 
We now apply Lemma \ref{lem:negative-basic-1} (i) to the case that $Z = X + \alpha$ and we have the assertion.
\end{proof}

\begin{Rem}
Let $\alpha \in \mathbb{H}$ and $\alpha \to 0$.
Then, formally,
$$E\left[ X^{\lambda} \right] = \frac{i^{\lambda}}{\Gamma(-\lambda)} \int_0^{\infty}  t^{-\lambda-1} \varphi_X (t) dt.$$ 
This is justified if $E\left[|X|^{\textup{Re}(\lambda)}\right] < +\infty$ and $t^{-\textup{Re}(\lambda)-1} \varphi_X(t) \in L^1$.
If $-1 < \textup{Re}(\lambda) < 0$ and $E\left[|X|^{\textup{Re}(\lambda)}\right] < +\infty$,
then, $\int_0^{\infty}  t^{-\lambda-1} \varphi_X (t) dt$ exists as a Riemann improper integral.
See \cite[Section 5]{Wolfe1975}.
The contour in the integration in \cite[Section 5]{Wolfe1975} is different from the contour in the integration in \eqref{eq:basic-neg}.
\end{Rem}

We now give an application to the power mean of random variables.

\begin{Thm}\label{thm:QAM-nega-upper}
Let $-1 \le p  < 0$.
Let $n \ge 2$.
Let $Z_1, Z_2, \dots, Z_n$ be i.i.d. $\mathbb{H}$-valued random variables such that $E\left[ \left(\textup{Im}(Z_1)^{1/p} |Z_1|^{1-1/p} \right)^{1/n}  \right] < +\infty$.
Then,
\[ \left.\frac{\partial^{1/p}}{\partial t^{1/p}} E\left[ \exp\left(i \frac{t}{n} Z_1^p\right) \right]^n\right|_{t=0}
= (-i)^{-1/p} E\left[\left( \frac{1}{n} \sum_{j=1}^{n} Z_j^p \right)^{1/p}\right].\]
\end{Thm}

\begin{proof}
Let $Z := \frac{1}{n} \sum_{j=1}^{n} Z_j^p$.
Since $-1 \le p < 0$, $Z$ is $(-\mathbb{H})$-valued.
We argue as in the proof of \cite[Proposition 3.3 (ii)]{Akaoka2022}.
By the geometric mean-arithmetic mean inequality,
\begin{equation}\label{eq:gm-hm}
E\left[ \left(-\textup{Im}(Z)\right)^{1/p}\right] \le E\left[ \left( -\textup{Im}(Z_1^p) \right)^{1/(np)}\right]^{n}.
\end{equation}

Let $r_1 > 0$ and $\theta_1 \in (0, \pi)$ such that $Z_1 = r_1 \exp(i\theta_1)$.
Then,
\begin{equation}\label{eq:polar-im}
\left(-\textup{Im}(Z_1^p) \right)^{1/p} = \textup{Im}(Z_1)^{1/p} |Z_1|^{1-1/p} \left(\frac{\sin\theta_1}{\sin(-p\theta_1)}\right)^{-1/p}.
\end{equation}
We see that $\sup_{\theta \in (0,\pi)} \left(\frac{\sin\theta}{\sin(-p\theta)}\right)^{-1/p} < +\infty$.
By this, \eqref{eq:polar-im}, and the assumption,
$E\left[ \left( -\textup{Im}(Z_1^p) \right)^{1/(np)}\right] < +\infty$, and hence, $E\left[(-\textup{Im}(Z))^{1/p}\right] < +\infty$.
Therefore we can apply Lemma \ref{lem:negative-basic-1} (ii), and we have the assertion.
\end{proof}

We now consider the continuity of the expectation of the power mean with respect to the parameter.

\begin{Prop}\label{prop:neg-conti-parameter}
Let $-1 < p_0  < 0$.
Let $n \ge 2$.
Let $Z_1, Z_2, \dots, Z_n$ be i.i.d. $\mathbb{H}$-valued random variables such that
$$\sup_{p \in (p_0 - \varepsilon_0, p_0 + \varepsilon_0)} E\left[ \left(\textup{Im}(Z_1)^{1/p} |Z_1|^{1-1/p} \right)^{1/n + \varepsilon_0}  \right] < +\infty$$
for some $\varepsilon_0 > 0$.
Then,
$E\left[\left( \frac{1}{n} \sum_{j=1}^{n} Z_j^p \right)^{1/p}\right]$ is continuous at $p = p_0$ as a function of $p$.
\end{Prop}

\begin{proof}
Let $-1 < p < 0$.
Let $Z^{(p)} := \frac{1}{n} \sum_{j=1}^{n} Z_j^p$.
Since
\[ E\left[\left( \frac{1}{n} \sum_{j=1}^{n} Z_j^p \right)^{1/p}\right] = \frac{1}{\Gamma(-1/p)} \int_0^{\infty} u^{-1/p - 1} E\left[ \exp(-iu Z^{(p)}) \right] du, \]
it suffices to show the continuity of $\int_0^{\infty} u^{-1/p - 1} E\left[ \exp(-iu Z^{(p)}) \right] du$ at $p = p_0$.
By Fubini's theorem,
\[ \int_0^{\infty} u^{-1/p - 1} E\left[\exp\left(u \textup{Im}(Z^{(p)})\right)\right] du = \Gamma\left(-\frac{1}{p}\right) E\left[ \left(-\textup{Im}(Z^{(p)}) \right)^{-1/p}\right]. \]
By \eqref{eq:gm-hm}, \eqref{eq:polar-im}, and the assumption, 
$E\left[ \left(-\textup{Im}(Z^{(p)}) \right)^{-1/p}\right]$ and $\int_0^{\infty} u^{-1/p - 1} E\left[\exp\left(u \textup{Im}(Z^{(p)})\right)\right] du$ are continuous at $p = p_0$.
We remark that $\left| E\left[ \exp(-iu Z^{(p)}) \right] \right| \le E\left[\exp\left(u \textup{Im}(Z^{(p)}) \right)\right]$.
By the Lebesgue convergence theorem, $E\left[ \exp(-iu Z^{(p)}) \right]$ and $u^{-1/p - 1} E\left[ \exp(-iu Z^{(p)}) \right]$, are continuous with respect to $p$.
Now we can apply a generalized Lebesgue convergence theorem (see \cite[Theorem 19 in Chapter 4]{Royden2010} for example),
and we see that $\int_0^{\infty} u^{-1/p - 1} E\left[ \exp(-iu Z^{(p)}) \right] du$ is continuous at $p = p_0$.
\end{proof}

We state a result applicable to general distributions on $\mathbb R$ satisfying an integrability condition.

\begin{Cor}\label{cor:neg-gen-real}
Let $\alpha \in \mathbb H$.
Then, we have the following:\\
(i) Let $-1 \le p  < 0$.
Let $X_1, X_2, \dots, X_n$ be i.i.d. real-valued random variables such that $X_1 \in L^{(p-1)/(np)}$.
Then,
\[ \left.\frac{\partial^{1/p}}{\partial t^{1/p}} E\left[ \exp\left(i \frac{t}{n} (X_1 + \alpha)^p\right) \right]^n\right|_{t=0}
= (-i)^{1/p} E\left[\left( \frac{1}{n} \sum_{j=1}^{n} (X_j + \alpha)^p \right)^{1/p}\right].\]
(ii) Let $r > 2/n$.
Assume that $X_1, X_2, \dots, X_n$ be i.i.d. real-valued random variables such that $X_1 \in L^{r}$.
Then,
$E\left[\left( \frac{1}{n} \sum_{j=1}^{n} (X_j + \alpha)^p \right)^{1/p}\right]$ is continuous with respect to $p$ on $(-1, -\frac{1}{nr-1})$.
\end{Cor}

\begin{proof}
The proof is the same as in the proof of Theorem  \ref{thm:QAM-nega-upper}.
If $Z_{j} := X_j + \alpha$, then, $\textup{Im}(Z_1)^{1/p} = \textup{Im}(\alpha)^{1/p} < +\infty$.
By substituting this into \eqref{eq:polar-im},
we have assertion (i).
We show (ii).
Let $p \in (-1, -\frac{1}{nr-1})$.
Then, by the H\"older inequality, for sufficiently small $\varepsilon > 0$,
\[ \sup_{p^{\prime} \in (p-\varepsilon, p+\varepsilon)} E\left[ |X_1 + \alpha|^{(1-1/p^{\prime})(1/n + \varepsilon)}  \right] \le \sup_{s \in [0,1]} E\left[|X_1 + \alpha|^r\right]^{s} < +\infty. \]
Hence we can apply Proposition \ref{prop:neg-conti-parameter}.
\end{proof}

\begin{Rem}
If $p \in U$ and $p \notin \mathbb{R}$,
then, it does not necessarily hold that $\textup{Im}(Z_1^p) < 0$, $P$-a.s., even if $Z_1 \in \mathbb H$. 
The above proof does not apply to this case.
When we consider the power means, it is natural to consider that $p \in \mathbb{R}$.
One reason is that it does not satisfy \cite[Assumption 2.1]{Akaoka2022},
more specifically,
if we let $f(z) := z^p$, then, $f(\mathbb H)$ can be {\it non-convex}.
Indeed, $0 \notin f(\mathbb H)$, but on the other hand,
there exist two points $x_1 < 0 $ and $x_2 > 0$ such that $x_1, x_2 \in f(\mathbb H)$.
Due to the branch cut, $\log\left(\sum_{j=1}^{n} Z_j^p \right)$ is not continuous with respect to $p$.
The power means of complex orders are interesting but much harder to deal with.
\end{Rem}

\section{Fractional derivative of positive order}\label{sec:pos}

We adopt the following definition as a fractional derivative of positive order.
As in the case of negative order, we consider the fractional derivative of complex order.
Here the positive order means that the real part of the complex order is positive. 
We adopt Marchaud's fractional derivative (\cite{Marchaud1927}) instead of the Riemann-Liouville fractional derivative, because it involves the derivative after an integration in order to define the fractional operator. 
Recall that $V = \{\lambda \in \mathbb{C} : \textup{Re}(\lambda) > 0\}$.

\begin{Def}
Let $k$ be a non-negative integer. 
Let $\lambda \in V$ such that $\lambda = k + \delta$ and  $0 < \textup{Re}(\delta) < 1$. 
Let $f \in C^k \left((-\infty,0] \right)$.
Then,
\[ \left. \frac{\partial^{k+\delta}}{\partial t^{k+\delta}} f(t)\right|_{t=0} := \frac{\delta}{\Gamma(1-\delta)} \int_0^{\infty}  \frac{f^{(k)}(0) - f^{(k)}(-u)}{u^{1+\delta}} du. \]
\end{Def}

\begin{Lem}\label{lem:positive-basic-1}
Let $\lambda \in V$.
We assume that $\textup{Re}(\lambda) \notin \mathbb{N}$. \\ 
Then, we have the following assertions:\\
(i) Let $Z$ be an $\overline{\mathbb{H}}$-valued random variable.
Assume that $E\left[|Z|^{\textup{Re}(\lambda)} \right] < +\infty$.
Then,
\[ \left.\frac{\partial^{\lambda}}{\partial t^{\lambda}} E[\exp(-itZ)]\right|_{t=0} = (-i)^{\lambda} E\left[Z^{\lambda}\right]. \] \\
(ii) Let $Z$ be an $\overline{-\mathbb{H}}$-valued random variable.
Assume that $E\left[|Z|^{\textup{Re}(\lambda)} \right] < +\infty$.
Then,
\[ \left.\frac{\partial^{\lambda}}{\partial t^{\lambda}} E[\exp(itZ)]\right|_{t=0} = i^{\lambda} E\left[Z^{\lambda}\right]. \]
\end{Lem}

\begin{proof}
We show (i).
The proof of (ii) is similar to it.
Let $f(t) := E[\exp(-itZ)]$.
Let $k$ be the integer part of $\textup{Re}(\lambda)$ and let $\delta := \lambda - k$. 
Then, $0 < \textup{Re}(\delta) < 1$. 
By the assumption  that $E\left[|Z|^{\textup{Re}(\lambda)} \right] < +\infty$,
we see that
$f \in C^k ((-\infty,0])$  and furthermore $f^{(k)}(t) = (-i)^k E[Z^k \exp(-itZ)]$ for $t \le 0$.

It holds that
\[ \int_0^{\infty} \left|\frac{1 - \exp(iuz)}{u^{1+\delta}}\right| du = |z|^{\textup{Re}(\delta)} \int_0^{\infty} \frac{|1-\exp(it e^{i\theta})|}{t^{1+ \textup{Re}(\delta)}} dt, \ \ \ z = re^{i\theta}, \ \theta \in [0, \pi].\]

If $0 \le t \le 1$, then,
$\frac{|1-\exp(it e^{i\theta})|}{t^{1+ \textup{Re}(\delta)}} \le \frac{e}{t^{\textup{Re}(\delta)}}$.
If $t > 1$, then, $\frac{|1-\exp(it e^{i\theta})|}{t^{1+ \textup{Re}(\delta)}} \le \frac{2}{t^{1+\textup{Re}(\delta)}}$.
Hence, we see that
\[ \int_{\overline{\mathbb H}} \int_0^{\infty} |z|^{k} \left|\frac{1 - \exp(iuz)}{u^{1+\delta}}\right| du P^Z (dz) < +\infty. \]

Therefore we can use Fubini's theorem and we obtain that
\[ \left. \frac{\partial^{\lambda}}{\partial t^{\lambda}} E[\exp(-itZ)]\right|_{t=0} =  \frac{(-i)^k \delta}{\Gamma(1-\delta)} \int_0^{\infty} \frac{E[Z^k (1-\exp(iuZ))]}{u^{1+\lambda}} du \]
\[ =  \frac{(-i)^k \delta}{\Gamma(1-\delta)} E\left[ Z^k \int_0^{\infty} \frac{1-\exp(iuZ)}{u^{1+\delta}} du\right].\]

By the Cauchy integral theorem, 
\begin{align}\label{eq:basic-posi}
\int_0^{\infty} \frac{1-\exp(iuz)}{u^{1+\delta}} du &= z^{\delta} \int_{C_z} \frac{1-\exp(i\zeta)}{\zeta^{1+\delta}} d\zeta \notag\\
&= z^{\delta} i^{-\delta} \int_0^{\infty} \frac{1-e^{-t}}{t^{1+\delta}} dt = \frac{\Gamma(1-\delta)}{\delta} z^{\delta} i^{-\delta}.
\end{align}
Thus we have the assertion.
\end{proof}

We now give applications of Lemma \ref{lem:positive-basic-1} to the power mean of random variables.
By Lemma \ref{lem:positive-basic-1} (i), 
we immediately see that
\begin{Prop}\label{prop:integral-posi-fractional}
Let $\alpha \in \overline{\mathbb{H}}$ and $\lambda \in V$. 
Assume that $\textup{Re}(\lambda) \notin \mathbb{N}$ and that $X \in L^{\textup{Re}(\lambda)}$. 
Let $k$ be the integer part of $\textup{Re}(\lambda)$ and let $\delta := \lambda - k$.
Then,
\[ E[(X + \alpha)^{\lambda}] = \frac{i^{\delta} \delta}{\Gamma(1-\delta)} \int_0^{\infty}  \frac{E\left[ (X + \alpha)^k \left(1 - \exp(iu(X+\alpha)) \right) \right] }{u^{1+\delta}} du. \]
\end{Prop}

We now give an application to the power mean of random variables.
For $p = 0$,
we regard $\left(\frac{1}{n} \sum_{j=1}^{n} z_j^p \right)^{1/p}$ as the geometric mean $\prod_{j=1}^{n} z_j^{1/n}$ for $z_1,  \dots, z_n  \in \overline{\mathbb H}$.

\begin{Thm}\label{thm:QAM-posi-upper}
Let $0 < p \le 1$.
Let $Z_1, \dots Z_n$ be i.i.d. $\overline{\mathbb H}$-valued random variables such that $Z_1 \in L^{1}$.
Then,
\begin{equation}\label{eq:qam-eq-posi-upper}
\left.\frac{\partial^{1/p}}{\partial t^{1/p}} E\left[ \exp\left(-i \frac{t}{n} Z_1^p\right) \right]^n\right|_{t=0} = (-i)^{1/p} E\left[\left( \frac{1}{n} \sum_{j=1}^{n} Z_j^p \right)^{1/p}\right].
\end{equation}
Furthermore, $E\left[\left( \frac{1}{n} \sum_{j=1}^{n} Z_j^p \right)^{1/p}\right]$ is continuous with respect to $p$ on $[0,1]$.
\end{Thm}

If $p=1/m$ for some natural number $m$, then, $\frac{\partial^{1/p}}{\partial t^{1/p}} $ denotes the symbol for the ordinal $m$-th derivative. 
The proof of continuity with respect to the parameter does not involve fractional calculus. 

\begin{proof}
The case that $p=1/m$ for some natural number $m$ is just the $m$-th derivative of the characteristic function and the result is well-known \cite[Chapter 16]{Williams1991}. 
Let $Z := \frac{1}{n} \sum_{j=1}^{n} Z_j^p$.
Then, by noting $0 <  p < 1$, $Z \in \overline{\mathbb{H}}$.
By the convexity of $x \mapsto |x|^{1/p}$ and the assumption that $Z_1 \in L^1$, 
$E\left[|Z|^{1/p} \right] < +\infty$.
We can apply Lemma \ref{lem:positive-basic-1} (i) and we have \eqref{eq:qam-eq-posi-upper}.

We remark that $|Z_j| = (|Z_j|^p)^{1/p}$ and $1/p > 1$. 
Then, by the H\"older inequality,
\[ \left| \left( \frac{1}{n} \sum_{j=1}^{n} Z_j^p \right)^{1/p} \right| \le \left( \frac{1}{n} \sum_{j=1}^{n} |Z_j|^p \right)^{1/p} \le \frac{1}{n} \sum_{j=1}^{n} |Z_j|. \]
By the assumption, $\frac{1}{n} \sum_{j=1}^{n} |Z_j|$ is integrable.
We remark that
\[ \lim_{p \to +0} \left( \frac{1}{n} \sum_{j=1}^{n} Z_j^p \right)^{1/p} = \lim_{p \to +0} \exp\left(\frac{1}{p} \log\left(\frac{1}{n} \sum_{j=1}^{n} Z_j^p \right) \right) \]
\[= \exp\left( \left.\frac{d}{dp} \log\left(\frac{1}{n} \sum_{j=1}^{n} Z_j^p \right) \right|_{p=0} \right) = \exp\left(\frac{1}{n} \sum_{j=1}^{n} \log Z_j \right). \]
The continuity assertion follows from the Lebesgue convergence theorem.
\end{proof}

As in the negative case, we state a result applicable to general distributions on $\mathbb R$ satisfying an integrability condition.
We see that
\begin{Cor}\label{cor:QAM-posi-upper}
Let $0 < p < 1$.
Let $\alpha \in \overline{\mathbb{H}}$.
Let $X_1, \dots, X_n$ be i.i.d. $\overline{\mathbb H}$-valued random variables such that $X_1 \in L^{1}$. 
Then, $\left( \frac{1}{n} \sum_{j=1}^{n} (X_j + \alpha)^p \right)^{1/p} \in L^1$ and
\[ \left.\frac{\partial^{1/p}}{\partial t^{1/p}} E\left[ \exp\left(-i \frac{t}{n} (X_1+\alpha)^p\right) \right]^n\right|_{t=0}
= (-i)^{1/p} E\left[\left( \frac{1}{n} \sum_{j=1}^{n} (X_j + \alpha)^p \right)^{1/p}\right].\]
Furthermore, $E\left[\left( \frac{1}{n} \sum_{j=1}^{n} (X_j + \alpha)^p \right)^{1/p}\right]$ is continuous with respect to $p$ on $[0,1]$.
\end{Cor}

If $X_1 \notin L^1$, then, $\left( \frac{1}{n} \sum_{j=1}^{n} (X_j + \alpha)^p \right)^{1/p} \notin L^1$.
See \cite[Proposition 5.1 (ii)]{Akaoka2022}.

\section{Characterizations of distributions on $\mathbb R$}\label{sec:char}

Throught this section, we use $\lambda$ as a symbol of exponents of complex-valued random variables. 
We first consider the case that $\textup{Re}(\lambda)  < 0$. 
Recall that $U = i \mathbb{H} = \{\lambda \in \mathbb{C} | \textup{Re}(\lambda) < 0\}$. 

We consider the analyticity of the moment with respect to the exponent $\lambda$. 
Let $U_{M} := \{z \in U : \textup{Re}(z) > -M\}$ for $M > 0$.

\begin{Lem}\label{lem:hol-fractional-power}
Let $Z$ be a complex-valued random variable.
Assume that $E\left[\left|\textup{Im}(Z)\right|^{-M}\right] < +\infty$ and $P(Z \in \mathbb{R}) = 0$.
Then, the map $\lambda \mapsto E\left[Z^{\lambda}\right]$ is holomorphic on $U_{M}$ and furthermore,
$$\frac{d}{d\lambda} E[Z^{\lambda}] = E\left[ Z^{\lambda} \log Z\right].$$ 
\end{Lem}

\begin{proof}
Let $C$ be a Jordan curve in $U_{M}$. 
By \eqref{eq:complex-power-abs},
$E\left[Z^{\lambda}\right]$ is continuous on $U_{M}$ and $E\left[ \sup_{\lambda \in C} |Z^{\lambda}|\right] \le E\left[\left|\textup{Im}(Z)\right|^{-M}\right] < +\infty$.
By Fubini's theorem and the Cauchy integral theorem,
$$\int_C E[Z^{\lambda}] d\lambda = E\left[ \int_C Z^{\lambda} d\lambda \right] = 0.$$ 
By Morera's theorem,
$E\left[Z^{\lambda}\right]$ is holomorphic on $U_{M}$.

By the Cauchy integral formula and Fubini's theorem,
we see that
\[ \frac{d}{d\lambda} E[Z^{\lambda}] = \frac{1}{2\pi i} \int_C \frac{E[Z^{\zeta}]}{(\zeta - \lambda)^2} d\zeta =  E\left[ \frac{1}{2\pi i} \int_C \frac{Z^{\zeta}}{(\zeta - \lambda)^2} d\zeta \right] =  E\left[ Z^{\lambda} \log Z\right]. \]
where we let $C$ be a circle with center $\lambda$ contained in $U_M$.
\end{proof}

There are several different proofs of this lemma.
One alternative way is to use the integral expression of $E\left[Z^{\lambda}\right]$ obtained from Lemma \ref{lem:negative-basic-1}.
By \eqref{eq:complex-power-abs}, we can also show the following assertion in the same manner as in the above lemma. 

\begin{Lem}\label{lem:hol-fractional-auxiliary}
Let $\lambda \in U$.
Let $Z$ be a $\overline{\mathbb{H}}$-valued random variable.
Then, the map $\alpha \mapsto E\left[(Z+\alpha)^{\lambda}\right]$ is holomorphic on $\mathbb H$  and furthermore,
$$\frac{d}{d\alpha} E\left[(Z+\alpha)^{\lambda} \right] = \lambda E\left[ (Z+\alpha)^{\lambda-1}\right].$$ 
\end{Lem}

Let $\mathcal{P}(\mathbb R)$ be the set of Borel probability measures on $\mathbb{R}$.
Let
$$F_{\eta}(\alpha, \lambda) := \int_{\mathbb R} (x+\alpha)^{\lambda} \eta(dx), \ \alpha, \lambda \in \mathbb{C}, \  \eta \in \mathcal{P}(\mathbb R),$$ 
if the integral exists.
This is a function of two variates.
Hereafter we {\it fix} either of the two variables.
Let
$$\mathcal{F}_{(\cdot, \lambda)} := \left\{F_{\eta} (\cdot, \lambda) | \eta \in \mathcal{P}(\mathbb{R}) \right\}$$ 
and 
$$\mathcal{F}_{(\alpha, \cdot)} := \left\{F_{\eta} (\alpha, \cdot) | \eta \in \mathcal{P}(\mathbb{R}) \right\}$$  
for $\lambda, \alpha \in \mathbb{C}$.
Let the characteristic function of $\eta \in \mathcal{P}(\mathbb R)$ be
$$\varphi_{\eta} (t) := \int_{\mathbb R} \exp(itx) \eta(dx), \ t \in \mathbb{R}.$$ 

We remark that for every $P \in \mathcal{P}(\mathbb{R})$, $\alpha \in \mathbb{H}$, and $\lambda \in U$,
the integral $F_{P}(\alpha, \lambda)$ is well-defined.
By Lemmas \ref{lem:hol-fractional-power} and  \ref{lem:hol-fractional-auxiliary},
all elements of $\mathcal{F}_{(\cdot, \lambda)}$ and $\mathcal{F}_{(\alpha, \cdot)}$ are holomorphic functions on $\mathbb{H}$ and $U$, respectively.

Let $D$ be a set and $\mathcal{F}$ be a class of complex-valued functions on $D$.
Then, we say that a subset $S$ of $D$ is a {\it determining set} of $(D, \mathcal{F})$ if it holds the property that if $f, g \in \mathcal{F}$ and $f(x) = g(x)$ for every $x \in S$ then $f(x) = g(x)$ for every $x \in D$.

\begin{Exa}\label{exa:determine}
(i) Let $D$ be a domain and $\mathcal{F}$ be the set of holomorphic functions on $D$.
If $S$ is a subset of $D$ which has an accumulating point in $D$, then, $S$ is a determining set of $(D, \mathcal{F})$, by the identity theorem for holomorphic functions. \\
(ii) Let $\mathbb{D} := \{z \in \mathbb{C} : |z| < 1\}$.
For $a > 0$, let $\mathbb{H}_a := \{z \in \mathbb{H} : \textup{Im}(z) > a\}$ and $\varphi_a(z) := \dfrac{z-(a+1)i}{z-(a-1)i}, \ z \in \mathbb{H}_a$.
The map $\varphi_a : \mathbb{H}_a \to \mathbb{D}$ is bi-holomorphic.
Let $(z_n)_{n \ge 1}$ be a sequence in $\mathbb{H}_a$ such that $\sum_{n=1}^{\infty} (1 - |\varphi_a (z_n)|) = +\infty$. 
This means that the {\it Blaschke condition} fails for $(\varphi_a (z_n))_{n}$.
For $\lambda \in (-\infty, 0)$ and $P \in \mathcal{P}(\mathbb{R})$,
the map $w \mapsto F_P \left( \varphi_a^{-1}(w), \lambda \right)$ is holomorphic and {\it bounded} on $\mathbb{D}$, since $a > 0$.
By \cite[Theorem 15.23]{Rudin1987},
the set $\{z_n\}_n$ is a determining set of $\left(\mathbb{H}_a, \mathcal{F}_{(\cdot, \lambda)} \right)$.

For example, if $|\varphi_{1} (z_n)| = 1 - 1/n$ for each $n$,
then, $\{z_n\}_n$ is an unbounded determining set of $\left(\mathbb{H}_{1}, \mathcal{F}_{(\cdot, \lambda)} \right)$. 
Since $\mathbb{H}_{1} \subset \mathbb{H}$ and $\mathcal{F}_{(\cdot, \lambda)}$ is a class of holomorphic functions on $\mathbb{H}$, by the identity theorem for holomorphic functions,
$\{z_n\}_n$ is an unbounded determining set of $\left(\mathbb{H}, \mathcal{F}_{(\cdot, \lambda)} \right)$. \\ 
(iii) Let $(\lambda_n)_{n \ge 1}$ be a sequence in $U$ such that $\sum_{n \ge 1} 1/(-\textup{Re}(\lambda_n)) = +\infty$ and $\sup_{n \ge 1} |\textup{Im}(\lambda_n)| < +\infty$. 
Then, by following the proof\footnote{The Blaschke condition is crucial in it.} of M\"untz-Szasz theorem in \cite[Theorem 15.26]{Rudin1987},
we see that $\{\lambda_n\}_{n \ge 1}$ is a determining set of $(U, \mathcal{F}_{(\alpha, \cdot)})$ for every $\alpha \in \mathbb{H}$.
\end{Exa}

Our motivation for introducing the notion of determining sets is due to Example \ref{exa:determine} (ii), in which the determining set is a {\it divergent} sequence in $\mathbb{H}$. 
We consider a characterization of distributions.

\begin{Thm}\label{thm:char-R-nega}
Let $\mu, \nu \in \mathcal{P}(\mathbb R)$.
Then,
$\mu = \nu$ if either of the following two conditions holds: \\
(i) There exist a point $\lambda \in U$ and a determining set $D_{\lambda}$ of $(\mathbb{H}, \mathcal{F}_{(\cdot, \lambda)})$ such that $F_{\mu} (\alpha,\lambda) = F_{\nu} (\alpha,\lambda)$ for every $\alpha \in D_{\lambda}$.\\
(ii) There exist a point $\alpha \in \mathbb{H}$ and a determining set $D_{\alpha}$ of $(U, \mathcal{F}_{(\alpha, \cdot)})$ such that $F_{\mu}(\alpha, \lambda) = F_{\nu}(\alpha,\lambda)$ for every $\lambda \in D_{\alpha}$.
\end{Thm}

We do not need any moment assumptions for $\mu$ and $\nu$.
It is, for example, applicable to the Cauchy distributions.
This is an extension of \cite[Theorem 2.1]{Okamura2020},  which is specific to the Cauchy distribution.
The following proof is completely different from the proof of \cite[Theorem 2.1]{Okamura2020}, which depends on the Riesz-Markov-Kakutani theorem.

\begin{proof}
Assume that (i) holds.
By the assumption and Lemma \ref{lem:hol-fractional-auxiliary},
$F_{\mu}(\alpha, \lambda) = F_{\nu}(\alpha, \lambda)$ for  every  $\alpha \in \mathbb H$, in particular on the imaginary axis in $\mathbb H$.
By Proposition \ref{prop:integral-neg-fractional}, 
$$\int_0^{\infty}  t^{-\lambda-1} \varphi_{\mu} (t) \exp(-st) dt = \int_0^{\infty}  t^{-\lambda-1} \varphi_{\nu} (t) \exp(-s t) dt, \ s > 0.$$ 
By an inversion formula for the Laplace transform,
$\varphi_{\mu} (t) = \varphi_{\nu} (t)$ for every $t > 0$.
Hence,
$\varphi_{\mu} (t) = \varphi_{\nu} (t)$ for every $t \in \mathbb R$.
By the L\'evy inversion formula, $\mu = \nu$. 

Assume that (ii) holds.
The proof is identical with the above one, except for using the inversion formula for the Mellin transform.
By the assumption and Lemma \ref{lem:hol-fractional-power},
$F_{\mu}(\alpha, \lambda) = F_{\nu}(\alpha, \lambda)$, for  every $\lambda \in (-\infty, 0)$.
By Proposition \ref{prop:integral-neg-fractional},
$$\int_0^{\infty}  t^{-\lambda-1} \varphi_{\mu} (t) \exp(i \alpha t) dt = \int_0^{\infty}  t^{-\lambda-1} \varphi_{\nu} (t) \exp(i \alpha t) dt, \ \lambda \in (-\infty, 0).$$ 
We remark that
$$\left|(\varphi_{\mu} (t) - \varphi_{\nu} (t)) \exp(i \alpha t)\right| \le 2\exp(- t \textup{Im}(\alpha)).$$  
By an inversion formula for the Mellin transform,
$\varphi_{\mu} (t) = \varphi_{\nu} (t)$ for every $t > 0$.
Hence, $\mu = \nu$.
\end{proof}

It is natural to regard $F_{\eta}(\alpha, \lambda)$ as a two-variate function on $\mathbb{H} \times U$.
By Hartogs' theorem, $F_{\eta}$ is a holomorphic function on $\mathbb{H} \times U$.
For holomorphic functions on $\mathbb{H} \times U$, every non-empty open subset is a determining set, however, a subset having an accumulating point in $\mathbb{H} \times U$ may not be a determining set.
The case of non-empty open subsets is attributed to the univariate case.
See \cite{Rudin2008} for discussions of determining sets.
We do not discuss this issue further.

We second consider the case that $\textup{Re}(\lambda) > 0$. 
Recall that $V = \{\lambda \in \mathbb{C} | \textup{Re}(\lambda) > 0\}$. 
Let $V_M := \{\lambda \in V | \textup{Re}(\lambda) < M\}$ for $M > 0$.
We let $z^{\lambda} \log z := 0$ if $z = 0$ and $\lambda \in V$.
We can also show the following in the same manner as in the proof of Lemma \ref{lem:hol-fractional-power}.

\begin{Lem}\label{lem:hol-frac-power-posi}
Let $Z$ be a complex-valued random variable.
Assume that $E\left[|Z|^M \right] < +\infty$.
Then, the map $\lambda \mapsto E\left[Z^{\lambda}\right]$ is holomorphic on $V_M$ and furthermore,
$$\frac{d}{d\lambda} E\left[Z^{\lambda}\right] = E\left[ Z^{\lambda} \log Z\right].$$  
\end{Lem}

There are some differences from Lemma \ref{lem:hol-fractional-power}.
In the above assertion, we need to assume the integrability condition for $Z$.
However, on the other hand, $Z$ can take real numbers with positive probability.

We also have the following assertion which corresponds to Lemma \ref{lem:hol-fractional-auxiliary}.
\begin{Lem}\label{lem:hol-frac-auxiliary-posi}
Let $\lambda \in V$.
Let $Z$ be a $\overline{\mathbb{H}}$-valued random variable  such that $E\left[|Z|^{\textup{Re}(\lambda)} \right] < +\infty$.
Then, the map $\alpha \mapsto E\left[(Z+\alpha)^{\lambda}\right]$ is holomorphic on $\mathbb H$ and furthermore,
\begin{equation}\label{eq:posi-plus-deri}
\frac{d}{d\alpha} E\left[(Z+\alpha)^{\lambda} \right] = \lambda E\left[ (Z+\alpha)^{\lambda-1}\right].
\end{equation}
\end{Lem}

The following results deal with the positive moment case.
We assume a moment condition for distributions.

\begin{Thm}\label{thm:char-R-posi}
Let $M > 0$.
Let $\mu, \nu \in \mathcal{P}(\mathbb R)$ such that 
$$\int_{\mathbb{R}} |x|^M \mu(dx) + \int_{\mathbb{R}} |x|^M \nu(dx) < +\infty.$$
Then,
$\mu = \nu$ if either of the following two conditions holds: \\
(i) There exist a point $\lambda \in V_M$ and a determining set $D_{\lambda}$ of $\left(\mathbb{H}, \mathcal{F}_{(\cdot,\lambda)} \right)$ such that $\textup{Re}(\lambda) \notin \mathbb{N}$ and $F_{\mu} (\alpha,\lambda) = F_{\nu} (\alpha,\lambda)$ for every $\alpha \in D_{\lambda}$.\\
(ii) There exist a point $\alpha \in \overline{\mathbb{H}}$ and a determining set $D_{\alpha}$ of $\left(V_M, \mathcal{F}_{(\alpha,\cdot)}\right)$such that $F_{\mu}(\alpha, \lambda) = F_{\nu}(\alpha,\lambda)$ for every $\lambda \in D_{\alpha}$.
\end{Thm}

This is an extension of \cite[Theorem 3.1]{Okamura2020},  which is specific to the Cauchy distribution.
By Lemmas \ref{lem:hol-frac-power-posi} and \ref{lem:hol-frac-auxiliary-posi},
we can apply Example \ref{exa:determine} as examples of the determining sets.
The following proof is similar to that of Theorem \ref{thm:char-R-nega}, but is a little more involved and largely different from the strategy taken in the proof of \cite[Theorem 3.1]{Okamura2020}.

\begin{proof}
Assume that (i) holds.
By  the assumption  and Lemma \ref{lem:hol-frac-auxiliary-posi},
\begin{equation}\label{eq:id-before-derivative}
 \int_{\mathbb R} (x+\alpha)^{\lambda} \mu(dx) = \int_{\mathbb R} (x+\alpha)^{\lambda} \nu(dx)
\end{equation}
for every $\alpha \in \mathbb{H}$.
Let $k$ be the integer part of $\textup{Re}(\lambda)$ and $\delta = \textup{Re}(\lambda)-k$.
By the assumption, $0 <  \textup{Re}(\delta) < 1$.
Recall \eqref{eq:posi-plus-deri}.
By differentiating the two expectations in \eqref{eq:id-before-derivative}
with respect to $\alpha$ $k$ times,
it  holds that
\[ \int_{\mathbb R} (x+\alpha)^{\delta} \mu(dx) = \int_{\mathbb R} (x+\alpha)^{\delta} \nu(dx) \]
for every $\alpha \in \mathbb{H}$.
By Proposition \ref{prop:integral-posi-fractional},
we see that
$$\int_0^{\infty} \frac{1-\varphi_{\mu}(t)\exp(it\alpha)}{t^{1+\delta}} dt = \int_0^{\infty} \frac{1-\varphi_{\nu}(t)\exp(it\alpha)}{t^{1+\delta}} dt, \ \alpha \in \mathbb{H}.$$
By differentiating these two integrals with respect to $\alpha$,
it holds that
$$\int_0^{\infty} \frac{\varphi_{\mu}(t)\exp(it\alpha)}{t^{\delta}} dt = \int_0^{\infty} \frac{\varphi_{\nu}(t)\exp(it\alpha)}{t^{\delta}} dt, \ \alpha \in \mathbb{H},$$  in particular on the imaginary axis in $\mathbb H$.
Hence,
$$\int_0^{\infty} \frac{\varphi_{\mu}(t)}{t^{\delta}} \exp(-tx) dt = \int_0^{\infty} \frac{\varphi_{\mu}(t)}{t^{\delta}} \exp(-tx)  dt, \ x > 0.$$ 
Now by the uniqueness of the Laplace transform,
$\varphi_{\mu}(t) = \varphi_{\nu}(t)$ for every $t > 0$, and hence, $\mu = \nu$.

Assume that (ii) holds.
By the assumption and Lemma \ref{lem:hol-frac-power-posi},
we see that
$F_{\mu}(\alpha, \lambda) = F_{\nu}(\alpha,\lambda)$ for every $\lambda \in V_M$.

By Proposition \ref{prop:integral-posi-fractional},
we see that
$$\int_0^{\infty} \frac{1-\varphi_{\mu}(t)\exp(it\alpha)}{t^{1+\delta}} dt = \int_0^{\infty} \frac{1-\varphi_{\nu}(t)\exp(it\alpha)}{t^{1+\delta}} dt$$ 
for every $\delta \in  V$ such that $0 < \textup{Re}(\delta) < \min\{1, M\}$.
By an inversion formula for the Mellin transform,
$\varphi_{\mu} (t) = \varphi_{\nu} (t)$ for every $t > 0$.
Hence, $\mu = \nu$.
\end{proof}

\begin{Rem}
(i) In condition (i) in Theorem \ref{thm:char-R-posi},
we need the assumption that $\textup{Re}(\lambda) \notin \mathbb{N}$.
If $\lambda = \textup{Re}(\lambda) = 1$, then, $E[(X + \alpha)^{\lambda}] = E[X] + \alpha$,  so even if we move the value of $\alpha$, we only know about the value of $E[X]$ and we cannot identify the distribution of $X$. \\
(ii) In condition (ii) in Theorem \ref{thm:char-R-posi},
$\alpha$ can take a real number. \\
(iii) The approach taken in the proof of \cite[Theorem 3.1]{Okamura2020} was similar to that of \cite[Theorem 1]{Lin1992}.
However, we cannot take these approaches here.
\end{Rem}

\section{Applications}\label{sec:appl}

In this section, we compute the expectations of the power means of the Cauchy distribution, the $t$-distribution whose degree of freedom is $3$, and the Poincar\'e distribution.
We first give an informal and heuristic argument.
Let $(Z_n)_n$ be i.i.d. $\mathbb{H}$-valued random variables. 
We will show that there exists a constant $\beta$ and an interval $I$ such that $E\left[\exp(i t Z_1) \right] = \exp(i \beta t)$ for every $t \in I$. 
Then, by the fractional derivative or complex analysis, $E[Z_1^p] = \beta^p$ for every $p$.
By the Taylor expansion, $E[\exp(itZ_1^p)] = \exp(i \beta^p t)$.
Finally, by the fractional derivative again, we see that $E\left[ \left( \frac{1}{n} \sum_{j=1}^{n} Z_j^p \right)^{1/p} \right] = \beta$ for every $n$ and $p$.
We need some modifications in the cases of Cauchy distributions.

\subsection{Cauchy distributions}

We deal with the Cauchy distribution with location $\mu$ and scale $\sigma$.
For $\mu \in \mathbb{R}$ and $\sigma > 0$, 
its probability density function $p(x)$ is given by
\[ p(x) = \frac{\sigma}{\pi} \frac{1}{(x-\mu)^2 + \sigma^2}, \ x \in \mathbb{R}. \]

Let $\gamma := \mu + \sigma i$.

\begin{Thm}[{\cite[Theorem 7.1]{Akaoka2022}}]\label{thm:neg-Cauchy-unbiased}
Let $\alpha \in \mathbb H$ and $-1 \le p  < 0$.
Let $n \ge 2$.
Let $X_1, \dots, X_n$ be i.i.d. random variables following the Cauchy distribution with location $\mu$ and scale $\sigma$.
Then,
$$E\left[\left( \frac{1}{n} \sum_{j=1}^{n} (X_j + \alpha)^p \right)^{1/p}\right] = \gamma + \alpha.$$ 
\end{Thm}

In \cite{Akaoka2022}, this assertion is shown by a repeated use of the Cauchy integral formula. 
Here, by Corollary \ref{cor:neg-gen-real}, we can naturally anticipate that the assertion holds. 
However, we cannot apply Corollary \ref{cor:neg-gen-real} directly, because $(p-1)/(np) < 1$ fails for some $n$.

\begin{proof}
Let $Z := \frac{1}{n} \sum_{j=1}^{n} (X_j + \alpha)^p$.
Let $P_M := P( \cdot | X_1 \le M)$ for $M > 0$.
We remark that $\textup{Im}((x+\alpha)^p) \le 0$ for every $x \in \mathbb{R}$.
Then, there exists a positive constant $c_M$ such that if $X_1 \le M$,
\[ \textup{Im}(Z) \le \frac{\textup{Im}((X_1 + \alpha)^p)}{n} \le - c_M < 0. \]
Therefore we can apply Lemma \ref{lem:negative-basic-1} (ii) and we obtain that
\[ \left.\frac{\partial^{1/p}}{\partial t^{1/p}} E_M \left[\exp(itZ)\right]\right|_{t=0} = i^{-1/p} E_M \left[Z^{1/p}\right], \]
where $E_M$ is the expectation with respect to $P_M$.
This means that
\begin{equation}\label{eq:EM-lim-1}
\frac{1}{\Gamma(-1/p)} \int_0^{\infty} t^{-1/p - 1} E_M [\exp(-itZ)] dt = i^{-1/p} E_M \left[Z^{1/p}\right].
\end{equation}

By the Lebesgue convergence theorem,
\begin{equation}\label{eq:EM-lim-2}
\lim_{M \to +\infty} E_M \left[Z^{1/p}\right] = E\left[Z^{1/p}\right],
\end{equation}
and
\begin{equation}\label{eq:EM-lim-3}
\lim_{M \to +\infty} E_M \left[\exp\left(i \frac{t}{n} (X_1 +\alpha)^p\right)\right] = E \left[\exp\left(i \frac{t}{n} (X_1+\alpha)^p\right)\right].
\end{equation}

By the Cauchy integral formula\footnote{As an alternative proof, we can also show the equality by using the Taylor expansion. See the proof of Theorem \ref{thm:Poincare} below. The alternative proof uses $E[(X_1 + \alpha)^{r}] = (\gamma + \alpha)^{r}$ for every $r < 0$.},
for every $t \le 0$,
\[ E\left[\exp\left(i t (X_1+\alpha)^p\right)\right] = E\left[\exp\left(i t Z\right)\right] = \exp\left(i t (\gamma + \alpha)^p\right). \]

By this, \eqref{eq:EM-lim-3} and
$$E_M \left[\exp(itZ)\right] = \exp\left(i \frac{n-1}{n} t (\gamma+\alpha)^p\right) E_M \left[\exp\left(i \frac{t}{n} (X_1+\alpha)^p\right)\right],$$  
we see that
\begin{equation}\label{eq:EM-lim-4}
\lim_{M \to \infty} E_M \left[\exp(itZ)\right] = \exp\left(i  t (\gamma+\alpha)^p\right).
\end{equation}

Since
$$\left|E_M \left[\exp\left(i \frac{t}{n} (X_1+\alpha)^p\right)\right]\right| \le 1,$$ 
we see that
$$\left| u^{-1/p - 1} E_M [\exp(-itZ)] \right| \le t^{-1/p-1} \exp\left(t \sin (p\theta_0) \right),$$ 
where $\arg(\gamma+\alpha) = \theta_0 \in (0, \pi)$.

By this and \eqref{eq:EM-lim-4},
we can apply the Lebesgue convergence theorem and obtain that
\begin{equation}\label{eq:EM-lim-5}
\lim_{M \to \infty} \int_0^{\infty} t^{-1/p - 1} E_M [\exp(-itZ)] dt = \int_0^{\infty} t^{-1/p - 1} \exp\left(-i  t (\gamma+\alpha)^p\right) dt.
\end{equation}
By \eqref{eq:basic-neg},
\begin{equation}\label{eq:basic-neg-cor}
\int_0^{\infty} t^{-1/p - 1} \exp\left(-i  t (\gamma+\alpha)^p\right) dt = (\gamma + \alpha) i^{-1/p} \Gamma\left(-\frac{1}{p}\right).
\end{equation}

By \eqref{eq:EM-lim-1}, \eqref{eq:EM-lim-2}, \eqref{eq:EM-lim-5} and \eqref{eq:basic-neg-cor}, we have the assertion.
\end{proof}

\begin{Rem}
(i) Let $\varphi$ be the characteristic function of $X_1$.
Then, $\varphi(t) = \exp(i\gamma t)$ for every $t \ge 0$.
By Lemma \ref{lem:negative-basic-1} (i),
we have that $E\left[ (X_1 + \alpha)^p \right]^{1/p} = \gamma + \alpha$ for $p \in (-1,0)$ and $\alpha \in \mathbb{H}$.
This is an alternative derivation of \cite[(6.1)]{Akaoka2022}.\\
(ii) The {\it positive} power mean is more difficult since $X_1 \notin L^1$.
\cite[Theorem 2.2 (a)]{Laue1980} does not hold for an $\mathbb{H}$-valued random variable $Z  =  \frac{1}{n}\sum_{j=1}^{n} (X_j + \alpha)^p$.
Let $p > 1$.
Assume that $p \notin \mathbb{N}$.
Let $k$ be the integer part of $1/p$.
Then, $(\sum_{j=1}^{n} (X_j + \alpha)^p)^k \in L^1$, and,
$$\left.\frac{\partial^{1/p}}{\partial t^{1/p}} E[\exp(-itZ)] \right|_{t=0} = \left.\frac{\partial^{1/p}}{\partial t^{1/p}} \exp(-it(\gamma+\alpha)^p) \right|_{t=0} = (-i)^{1/p} (\gamma + \alpha).$$
However, $(\sum_{j=1}^{n} (X_j + \alpha)^p)^{1/p} \notin L^1$.
\end{Rem}

\subsection{$t$-distributions}

We consider the location-scale family of a slightly modified $t$-distribution whose degree of freedom is $3$.
Let $f(x) := \frac{2}{\pi} (1+x^2)^{-2}$ for $x \in \mathbb{R}$.
Let $\mu \in \mathbb{R}$ and $\sigma > 0$.
Let $\gamma := \mu + \sigma i$.
Let
\[ p(x) := \frac{1}{\sigma} f\left(\frac{x-\mu}{\sigma} \right) = \frac{2\sigma^3}{\pi} \frac{1}{|x-\gamma|^4}, \ x \in \mathbb{R}. \]
Let $X_1, X_2, \dots$ be i.i.d. random variables such that $X_1$ has a density function $p(x)$.
Let $\alpha \in \mathbb{H}$ and $p < 0$.

We can argue with this case in the same manner as in the case of Cauchy distributions, so we only give a sketch of arguments.
By the Cauchy integral formula,
it holds that
$$E\left[(X_1 + \alpha)^{p} \right] = (\gamma + \alpha)^{p-1} \left(\gamma + \alpha - i p \sigma\right)$$ 
and for $t \le 0$,
$$E\left[\exp\left(it(X_1 + \alpha)^{p}\right) \right] = \left(1+p \sigma (\gamma + \alpha)^{p-1} t\right) \exp\left(i (\gamma + \alpha)^p t\right).$$  

As before,
let $Z := \frac{1}{n} \sum_{j=1}^{n} (X_j + \alpha)^p$.
Then,
$$E\left[ \exp(itZ) \right] = \sum_{k=0}^{n} \binom{n}{k} \left(\frac{p(\gamma+\alpha)^{p-1}}{n}\right)^k t^k \exp\left(i (\gamma + \alpha)^p t\right).$$ 
By this and \eqref{eq:basic-neg},
\[ \int_{0}^{\infty} t^{-1/p-1} E\left[ \exp(itZ) \right] dt = i^{-1/p} \sum_{k=0}^{n} \binom{n}{k} \left(\frac{ip}{n(\gamma+\alpha)}\right)^k \Gamma\left(k - \frac{1}{p}\right).   \]

Thus we see that
\begin{equation}\label{eq:t3}
E\left[\left( \frac{1}{n} \sum_{j=1}^{n} (X_j + \alpha)^p \right)^{1/p}\right] = \frac{\gamma+\alpha}{\Gamma(-1/p)} \sum_{k=0}^{n} \binom{n}{k} \left( \frac{ip}{n(\gamma+\alpha)}\right)^k \Gamma\left(k - \frac{1}{p}\right). 
\end{equation}

Since $\dfrac{p^k \Gamma(k-1/p)}{\Gamma(-1/p)} = \prod_{j=0}^{k-1} (jp-1)$, 
the right hand side of \eqref{eq:t3} is a polynomial in $p$ of degree $n-1$ with complex coefficients. 
The coefficient of the highest degree is $\dfrac{(n-1)! i^n}{n^n (\gamma+\alpha)^{n-1}}$. 
Hence, $E\left[\left( \frac{1}{n} \sum_{j=1}^{n} (X_j + \alpha)^p \right)^{1/p}\right]$ is not a constant function with respect to $p$. 
The same arguments apply to the $ t$-distribution whose degree of freedom is an odd number, but the expression of the expectation would be more complicated.

\subsection{Poincar\'e distributions}

The probability density function of a Poincar\'e distribution with a parameter $\theta = (a,b,c)$ is given by
\begin{equation*}
p_{\theta}(x,y) :=  \frac{D\exp(2D)}{\pi} \exp\left(- \frac{a(x^2+y^2)+2bx+c}{y}\right) \frac{1}{y^2}, \ \ x \in \mathbb{R}, y > 0,
\end{equation*}
where  $\theta$ belongs to the parameter space
$\Theta :=\{(a,b,c)\in\mathbb{R}^3\ : a>0, c>0, \ ac-b^2>0 $ 
and we let $D := \sqrt{ac-b^2}$.

This is the upper-half plane realization of two-dimensional hyperboloid distributions and was introduced by \cite{tojo-2020}.
The family is compatible with the Poincar\'e metric on $\mathbb{H}$.
Properties of the Poincar\'e distribution have also been investigated in \cite{tojo2021harmonic,nielsen2023i}.

By direct computations, we see that
\begin{Prop}\label{prop:half-char-Poincare}
Let $Z$ be an $\mathbb{H}$-valued random variable following a Poincar\'e distribution with a parameter $(a,b,c) \in \Theta$.
Then, 
$$E[\exp(itZ)] = \exp\left(-i \frac{b}{a} t - \frac{D}{a} t\right), \ t \ge 0.$$
\end{Prop}

\begin{proof}
It holds that 
$$E[\exp(itZ)] = \int_{\mathbb H} \exp(itx - ty) p_{\theta}(x,y) dxdy.$$  
We first calculate the integration with respect to $x$ and we obtain that
\[  \int_{\mathbb R} \exp(itx) \exp\left(- \frac{a(x^2+y^2)+2bx+c}{y}\right) dx = \sqrt{\frac{\pi y}{a}}\exp\left(-i \frac{b}{a} t - \left(a + \frac{t^2}{4a} \right)y - \frac{D^2}{ay}\right).\]
We secondly integrate the above function with respect to $y$.
We see that
$$\int_0^{\infty} \frac{1}{y^{3/2}} \exp\left( - \frac{(t+2a)^2}{4a} y - \frac{D^2}{ay} \right) dy = \frac{\sqrt{\pi a}}{D} \exp\left(-\frac{D(t+2a)}{a}\right).$$
The assertion follows from this.
\end{proof}

\begin{Thm}\label{thm:Poincare}
Let $n \ge 2$.
Let $Z_1, \dots, Z_n$ be i.i.d. $\mathbb{H}$-valued random variables following a Poincar\'e distribution with a parameter $(a,b,c) \in \Theta$.
Then, \\
(i) For every $p$ with $p \ne 0$,
\begin{equation}\label{eq:power-poincare}
E\left[Z_1^p \right] = \left(-\frac{b}{a} + \frac{D}{a}i\right)^p.
\end{equation}
(ii) $E\left[\exp\left(t |Z_1|^p \right)\right] < +\infty$ for every $t > 0$ and $p$ with $0 < |p| < 1$.\\
(iii) For every $p$ with $0 < |p| \le 1$,
\begin{equation}\label{eq:poincare-unbiased}
 E\left[\left( \frac{1}{n} \sum_{j=1}^{n} Z_j^p \right)^{1/p}\right] = -\frac{b}{a} + \frac{D}{a}i.
 \end{equation}
\end{Thm}

For ease of notation, we let
$$I(a,b,c) :=  \int_{\mathbb H} \exp\left(- \frac{a(x^2+y^2)+2bx+c}{y}\right) \frac{1}{y^2} dxdy, \ \theta  = (a,b,c) \in \Theta.$$  
We show these assertions by dividing the cases according to the sign of $p$.

\begin{proof}[Proof of the case that $p  < 0$] 
(i)
We will show that for every $\lambda > 0$, $E[\textup{Im}(Z_1)^{-\lambda}] < +\infty$. 
We see that 
$$E[\textup{Im}(Z_1)^{-\lambda}]  \le \frac{D\exp(2D)}{\pi}\int_{\mathbb H} \frac{1}{y^{2+\lambda}} \exp\left(-\frac{a(x^2+y^2)+2bx+c}{y}\right) dxdy < +\infty,$$ 
because for every $\varepsilon > 0$,
$$\sup_{y > 0} y^{-2-\lambda} \exp\left(-\frac{\varepsilon}{y}\right) < +\infty,$$ 
and for sufficiently small $\varepsilon > 0$,
$I(a,b,c-\varepsilon) < +\infty$.

Therefore we can apply Lemma \ref{lem:negative-basic-1} (i).
By recalling Proposition \ref{prop:half-char-Poincare} and \eqref{eq:basic-neg},
we have \eqref{eq:power-poincare}.

(ii) We obtain that for every $t > 0$,
\[ E\left[\exp(t |Z_1|^p)\right] \le E\left[\exp(t \textup{Im}(Z_1)^p)\right] \]
\[ = \frac{D\exp(2D)}{\pi}\int_{\mathbb H} \frac{\exp(t y^p)}{y^{2}} \exp\left(-\frac{a(x^2+y^2)+2bx+c}{y}\right) dxdy < +\infty, \]
because for every $\varepsilon > 0$,
$$ \sup_{y > 0} \exp\left(t y^p - \frac{\varepsilon}{y}\right) < +\infty.$$  
and for sufficiently small $\varepsilon > 0$,
$I(a,b,c-\varepsilon) < +\infty$.

(iii) For $n \in \mathbb{N}$ and $p \in [-1,1]$,
it holds\footnote{However, it does {\it not} hold that $z^{np} = (z^n)^p$. } that $z^{np} = (z^p)^n$ for every $z \in \mathbb{C}$.
By the Taylor expansion, (i) and (ii),
we see that for every $u \ge 0$, 
\begin{equation}\label{eq:poincare-char-neg}
E\left[\exp(iuZ_1^p)\right] = \sum_{k=0}^{\infty} \frac{(iu)^k E\left[Z_1^{pk}\right]}{k!} = \sum_{k=0}^{\infty} \frac{(iu)^k \left(-\frac{b}{a} + \frac{D}{a}i\right)^{pk}}{k!} = \exp\left( iu \left(-\frac{b}{a} + \frac{D}{a}i \right)^p\right).
\end{equation}

We remark that
$$\int_{\mathbb H} \frac{(x^2 + y^2)^{(1-1/p)/2}}{y^{2-1/p}} \exp\left(-\frac{a(x^2+y^2)+2bx+c}{y}\right) dxdy < +\infty,$$ 
because for every $\varepsilon > 0$
$$\sup_{x+iy \in \mathbb H} (x^2 + y^2)^{(1-1/p)/2} \exp\left(-\varepsilon \frac{x^2 + y^2}{y}\right) < +\infty,$$
and,
$$\sup_{y > 0} y^{1/p} \exp\left(-\frac{\varepsilon}{y}\right) < +\infty,$$
and for sufficiently small $\varepsilon > 0$,
$I(a-\varepsilon,b,c-\varepsilon) < +\infty$.
Hence, 
$$E\left[ \textup{Im}(Z_1)^{1/p} |Z_1|^{1-1/p}  \right] < +\infty.$$
Let $Z := \frac{1}{n} \sum_{j=1}^{n} Z_j^p$.
By Theorem \ref{thm:QAM-nega-upper},
we see that
\[ \left.\frac{\partial^{1/p}}{\partial t^{1/p}} E\left[\exp\left(i\frac{t}{n}Z_1^p\right)\right]^n\right|_{t=0} = i^{-1/p} E\left[Z^{1/p}\right]. \]

By \eqref{eq:poincare-char-neg},
$$E\left[\exp\left(i\frac{t}{n}Z_1^p\right)\right]^n =  \exp\left( it \left(-\frac{b}{a} + \frac{D}{a}i \right)^p\right).$$ 
By this and \eqref{eq:basic-neg},
we have \eqref{eq:poincare-unbiased}.
\end{proof}

\begin{proof}[Proof of the case that $p>0$]
We first show that

\begin{Lem}\label{lem:exp-integrability}
$E\left[ \exp(t |Z_1|) \right] < +\infty$ for a sufficiently small $t > 0$.
\end{Lem}

\begin{proof}
Let $\varepsilon > 0$ such that $I(a-\varepsilon, b, c) < +\infty$.
Then, 
\[ E\left[\exp\left(\frac{\varepsilon}{2} |Z_1|\right)\right] \le \frac{D\exp(2D)}{\pi} I(a-\varepsilon, b, c) < +\infty, \]
because 
$$\sup_{x+yi \in \mathbb{H}} \exp\left(\frac{\varepsilon}{2} \sqrt{x^2 + y^2}  -\varepsilon\frac{x^2 + y^2}{y}\right) \le 1.$$
\end{proof}

(i)
If $p$ is an integer, then, we have \eqref{eq:power-poincare} by Proposition \ref{prop:half-char-Poincare}  and Lemma \ref{lem:exp-integrability}.
Assume that $p$ is not an integer.
By Lemma \ref{lem:exp-integrability},
$$E\left[\textup{Im}(Z_1)^{p}\right] \le E\left[|Z_1|^{p}\right] < \infty.$$
Therefore we can apply  Lemma \ref{lem:positive-basic-1} (i) and we obtain that
\[ \left. \frac{\partial^{p}}{\partial t^{p}} E[\exp(-itZ_1)]\right|_{t=0} = (-i)^{p} E\left[Z_1^{p}\right]. \]
By this, Proposition \ref{prop:half-char-Poincare}, and \eqref{eq:basic-neg},
we have \eqref{eq:power-poincare}. 

(ii) follows from Lemma \ref{lem:exp-integrability}.

(iii) The case of $p=1$ is easy.
Assume that $0 < p < 1$.
Let $Z := \frac{1}{n} \sum_{j=1}^{n} Z_j^p$.
By the Taylor expansion, (i) and (ii),
we see that for every $u \ge 0$,
\begin{equation}\label{eq:poincare-char-posi}
E\left[\exp(iuZ_1^p)\right] = \sum_{k=0}^{\infty} \frac{(iu)^k E\left[Z_1^{pk}\right]}{k!} =  \exp\left( iu \left(-\frac{b}{a} + \frac{D}{a}i \right)^p\right).
\end{equation}

By H\"older's inequality,
we see that
$$ E\left[\textup{Im}(Z)^{1/p}\right] \le E\left[|Z|^{1/p}\right] \le E\left[ \left(\frac{1}{n} \sum_{j=1}^{n} |Z_j|^p\right)^{1/p} \right]  \le E\left[ \frac{1}{n} \sum_{j=1}^{n} |Z_j| \right] = E\left[ |Z_1| \right] < +\infty.$$
Therefore we can apply  Theorem \ref{thm:QAM-posi-upper} and we obtain that
\[ \left.\frac{\partial^{1/p}}{\partial t^{1/p}} E \left[\exp\left(-i\frac{t}{n}Z_1^p\right)\right]^n\right|_{t=0} = (-i)^{p} E\left[Z^{1/p}\right]. \]
By \eqref{eq:poincare-char-posi},
$$E \left[\exp\left(-i\frac{t}{n}Z_1^p\right)\right]^n = \exp\left(- it \left(-\frac{b}{a} + \frac{D}{a}i \right)^p\right).$$  
By this and \eqref{eq:basic-posi}, we have \eqref{eq:poincare-unbiased}.
\end{proof}

\begin{Rem}
(i) The case that $p=0$ corresponds to the geometric mean $E\left[Z_1^{1/n} \right]^n$, and the value is $-\frac{b}{a} + \frac{D}{a}i$ by \eqref{eq:power-poincare}. \\
(ii) Assume that $Z$ and $W$ are $\mathbb H$-valued.
It can happen that $E[\exp(itZ)] = E[\exp(itW)]$ for every $t \ge 0$, but the distributions of $Z$ and $W$ are different, as in the above case.
If  $E[\exp(itZ)] = E[\exp(itW)]$ for every $t \ge 0$,
then, by Lemma \ref{lem:negative-basic-1} (i) and Lemma \ref{lem:positive-basic-1} (i),
$E[Z^p] = E[W^p]$ for every $p \in \mathbb R$. \\
(iii) According to numerical computations, \eqref{eq:poincare-unbiased} will fail if $|p| > 1$.
\end{Rem}

\appendix
\section{Comparison with fractional absolute moments}

It is natural to compare the fractional absolute moment $E[|Z^{p}|] = E[|Z|^p]$ with the absolute value of the fractional moment $|E[Z^{p}]|$ for $p \in \mathbb{R}$. 

\begin{Prop}\label{prop:frac-abs-compare}
Assume that $p$ is a real number with $|p| \le 1$.
Let $Z$ be a $\overline{\mathbb{H}}$-valued  random variable or a $-\overline{\mathbb{H}}$-valued random variable. 
Then,
\begin{equation}\label{eq:compare-frac-abs-1}
E\left[ |Z|^p \right] \le \frac{\left| E\left[ Z^p \right] \right|}{\cos (p \pi/2)}.
\end{equation}
\end{Prop}

\begin{proof}
The case that $p \in \{0, \pm1\}$ is easily seen.
Assume that $Z$ is $\overline{\mathbb{H}}$-valued and $0 < p < 1$.
The other cases are shown in the same manner.

Let $n \ge 1$ and $Z_1, \dots Z_n$ be i.i.d. copies of $Z$. 
Let $\xi_i := Z_i / |Z_i|$ if $Z_i \ne 0$ and $\xi_i := 1$ if $Z_i = 0$.
Let $s_i := |Z_i| / \sum_{i=1}^{n} |Z_i|$ if $\sum_{i=1}^{n} |Z_i| > 0$ and $s_i := 1/n$ if $\sum_{i=1}^{n} |Z_i| = 0$.
Then,
$$\sum_{i=1}^{n} Z_i^p = \left(\sum_{i=1}^{n} |Z_i|^p \right) \left( \sum_{i=1}^{n} s_i \xi_i^p  \right).$$ 
We remark that $\sum_{i=1}^{n} s_i \xi_i^p$ is in the convex hull of $\{\xi_i^p\}_{i=1}^{n}$ and the convex hull is contained in the region surrounded by the arc $\{\exp(i\theta) : 0 \le \theta \le p\pi\}$ and the line segment connecting $1$ and $(-1)^p = \exp(ip\pi)$.
Hence, $| \sum_{i=1}^{n} s_i \xi_i^p | \ge \cos(p \pi/2)$ and 
$$\left|\sum_{i=1}^{n} Z_i^p\right| \ge \left(\sum_{i=1}^{n} |Z_i|^p \right) \cos(p \pi/2).$$ 
Now the assertion follows from the law of large numbers.
\end{proof}

In the same manner, we see that
\begin{Prop}
Assume that $p$ is a real number with $|p| < 1/2$.
Let $Z$ be a complex-valued random variable.
Then,
$$E\left[ |Z|^p \right] \le \frac{\left| E\left[ Z^p \right] \right|}{\cos (p \pi)}.$$ 
\end{Prop}

\begin{Rem}
(i) If $|p| \le 1$ and $P(Z = 1) = P(Z = -1) = 1/2$, then,
$E[|Z|^p] = 1$ and $\left| E\left[ Z^p \right] \right| = |1+(-1)^p|/2 = \cos(p \pi/2)$.
Hence the bound in \eqref{eq:compare-frac-abs-1} is best in general. \\
(ii) If $1/2 < p \le 1$ and $P(Z = 1) = P(Z = \exp(i (1-1/p)\pi)) = 1/2$,
then, $E[Z^p] = 0$ and $E[|Z|^p] = 1$.
The case that $-1 \le p < 1/2$ is the same.
\end{Rem}

We give an application of Proposition \ref{prop:frac-abs-compare}.
Let $(Z_n)_n$ be $\overline{\mathbb{H}}$-valued i.i.d. random variables such that $Z_1 \in L^p$ for some $p > 0$.
Let $\ell \in \mathbb{N}$ such that $Z_1 \in L^{1/\ell}$.
Let $Y_j := Z_j^{1/(j+\ell)} / E\left[ Z_j^{1/(j+\ell)} \right]$, $j \ge 1$, and $M_n := \prod_{j=1}^{n} Y_j$, $n \ge 1$.
Let $\mathcal{F}_{n} := \sigma(Z_1, \dots, Z_n)$.
Then, $(M_n, \mathcal{F}_n)_n$ is a martingale.
By \eqref{eq:compare-frac-abs-1},
$$\sup_{n \ge 1} E\left[ |Z_n| \right] \le \prod_{j=1}^{\infty} \frac{1}{\cos(\frac{\pi}{2} (j+\ell))}\le C \sum_{j=1}^{\infty} \frac{1}{(j+\ell)^2} < +\infty.$$
By the martingale convergence theorem (\cite[Chapter 11]{Williams1991}),
$(Z_n)_n$ converges almost surely as $n \to \infty$.
By Kronecker's lemma,
$$ \prod_{j=1}^{n} \frac{Z_j^{1/n}}{E\left[ Z_j^{1/(j+\ell)} \right]^{(j+\ell)/n}} \to 1, \ n \to \infty, \ \textup{ a.s.} $$ 
Since $Z_1 \in L^{1/\ell}$,
$\prod_{j=1}^{n} Z_j^{1/n} \to \exp(E[\log Z_1])$, $n \to \infty$, a.s.
This gives a martingale proof of the strong law of large numbers {\it for the geometric means}. \\

\noindent{\it Acknowledgements}. \ 
The authors would like to give our sincere thanks to the reviewer for his or her careful reading of the manuscript and many comments. 
The authors thank Prof. Bin Xie for the discussions giving them the inspiration for applying fractional calculus. 
They also thank Prof. Masanori Adachi for his comments on determining sets. 
K.O. and Y.O. have been supported by JSPS KAKENHI Grant Numbers JP22K13928 and JP23K03213 respectively.

\end{document}